\documentclass[12pt]{amsart} 
\usepackage{hyperref}
\usepackage{color}
\usepackage{amsmath, amsthm,url, 
  amssymb,setspace,epsfig, mathrsfs,fontenc, 
  fullpage, MnSymbol}
\usepackage[alphabetic]{amsrefs}

\usepackage[T1]{fontenc}
\newenvironment{dedication}
  {
   \thispagestyle{empty}
   \vspace*{\stretch{1}}
   \itshape             
   \raggedleft          
  }
  {\par 
   \vspace{\stretch{3}} 
   \clearpage           
  }

\DefineSimpleKey{bib}{myurl}
\newcommand\myurl[1]{\url{#1}}
\BibSpec{webpage}{
  +{}{\PrintAuthors} {author}
  +{,}{ \textit} {title}
  +{}{ \parenthesize} {date}
  +{,}{ \myurl} {myurl}
}

\newtheorem{thm}{Theorem}
\newtheorem{lem}[thm]{Lemma}
\newtheorem{exam}[thm]{Example}
\newtheorem{prop}[thm]{Proposition}

\newtheorem{rem}[thm]{Remark}
\newtheorem{cor}[thm]{Corollary}
\newtheorem{defe}[thm]{Definition}

\newcommand{\nc}{\newcommand}
\nc{\renc}{\renewcommand}
\nc{\ssec}{\subsection}
\nc{\sssec}{\subsubsection}
\nc{\on}{\operatorname}
\nc {\ra} {\rightarrow}
\nc {\inj} {\hookrightarrow}
\nc {\lra} {\longleftrightarrow} 
\nc         {\rar}[1]       {\stackrel{#1}{\longrightarrow}}

\nc {\Sl}{\mathfrak{sl}}
\nc {\Gl}{\mathrm{GL}}
\nc{\fm}{\mathfrak{m}}
\nc{\fg}{\mathfrak{g}}
\nc{\fp}{\mathfrak{p}}
\nc{\ft}{\mathfrak{t}}
\nc{\fk}{\mathfrak{k}}
\nc{\fI}{\mathfrak{i}}
\nc{\fh}{\mathfrak{h}}
\nc{\fb}{\mathfrak{b}}
\nc{\fu}{\mathfrak{u}}
\nc{\fn}{\mathfrak{n}}
\nc{\hfg}{\widehat{\fg}}
\nc{\hfh}{\widehat{\fh}}

\nc {\hH}{{\check{H}}}
\nc {\hB}{{\check{B}}}
\nc {\hN}{{\check{N}}}
\nc {\hG}{{{\check{G}}}}
\nc {\cfg}{\check{\fg}}
\nc {\cfb}{\check{\fb}}
\nc {\cfn}{\check{\fn}}
\nc {\cfh}{\check{\fh}}
\nc {\clambda}{\check{\lambda}}

\nc {\bone}{\mathbf{1}}
\nc {\bu}{\mathbf{u}}
\nc {\bv}{\mathbf{v}}
\nc {\bw}{\mathbf{w}}

\nc{\bC}{\mathbb{C}}
\nc{\bM} {\mathbb{M}}
\nc{\bU} {\mathbb{U}}
\nc{\bV} {\mathbb{V}}
\nc{\bW}{\mathbb{W}}
\nc{\bL}{\mathbb{L}}
\nc {\bZ} {\mathbb{Z}}
\nc {\bGm} {\mathbb{G}_m}

\nc {\cD}{\mathcal{D}}
\nc {\cE}{\mathcal{E}}
\nc {\cF}{\mathcal{F}}
\nc {\cG}{\mathcal{G}}
\nc {\cK}{\mathcal{K}}
\nc{\cZ}{\mathcal{Z}}
\nc {\cDt}{\cD^\times} 
\nc {\cA} {\mathcal{A}}
\nc {\cI}{\mathcal{I}}
\nc {\cR}{\mathcal{R}}
\nc {\cJ}{\mathcal{J}}
\nc {\cS}{\mathcal{S}}
\nc {\cO}{\mathcal{O}}
\nc {\cP}{\mathfrak{P}}

\nc {\tU} {\widetilde{U}}

\nc {\Ind}{\mathrm{Ind}}
\nc {\GL}{\mathrm{GL}}
\nc {\RS}{\mathrm{RS}}
\nc {\Fun}{\mathrm{Fun}}
\nc {\Lie}{\mathrm{Lie}}
\nc {\Hom}{\mathrm{Hom}}
\nc {\Vac}{\mathrm{Vac}}
\nc {\Res}{\mathrm{Res}} 
\nc {\ev}{\mathrm{ev}}
\nc{\Spec}{\mathrm{Spec}\,}
\nc {\ord}{\mathrm{ord}}
\nc {\Op}{\mathrm{Op}}
\nc{\Loc}{\mathrm{Loc}}
\nc {\MOp}{\mathrm{MOp}}
\nc {\Conn}{\mathrm{Conn}}
\nc {\MT}{\mathrm{MT}}
\nc{\Sym}{\mathrm{Sym}}
\nc {\modd}{\mathrm{mod}}
\nc {\ad}{\mathrm{ad}}
\nc {\Tr}{\mathrm{Tr}}
\nc{\Kil}{\mathrm{Kil}}
\nc {\End}{\mathrm{End}}
\nc {\can}{\mathrm{can}}
\nc {\Gr}{\mathrm{Gr}}
\nc {\Aut}{\mathrm{Aut}}
\nc {\Der}{\mathrm{Der}}
\nc {\gr}{\mathrm{gr}}
\nc {\sm}{\mathrm{sm}}
\nc {\Map}{\mathrm{Map}}
\nc {\Hitch}{\mathrm{Hitch}}
\nc {\Higgs}{\mathrm{Higgs}}
\nc {\T}{\mathrm{T}}
\nc{\Bun}{\mathrm{Bun}}

\nc {\bP} {\bar{P}}
\nc {\bJ}{\bar{J}}

\usepackage[all]{xy}

\nc {\quo}{\mathopen{ /\!/}}
\nc {\llp} {\mathopen{ (\!(}}
\nc {\rrp} {\mathopen{ )\!)}}
\nc {\llb} {\mathopen{ [\![}}
\nc {\rrb} {\mathopen{ ]\!]}}
\nc {\lc} {\mathopen{:\!}}
\nc {\rc}{\mathopen{\!:}}
\nc{\fgbb} {\fg \llb t \rrb}
\nc{\fgpp} {\fg\llp t \rrp}
\nc{\fhbb} {\fh\llb t \rrb}
\nc{\fhpp} {\fh\llp t \rrp}
\nc{\bCpp} {\bC \llp t \rrp}
\nc{\bCbb} {\bC \llb t \rrb}
\nc{\hNpp} {\hN \llp t \rrp}
\nc{\hNbb} {\hN \llb t \rrb}
\nc{\cfbpp} {\cfb \llp t \rrp}
\nc{\cfbbb} {\cfb \llb t \rrb}

\nc {\Gpp} {G\llp t \rrp}
\nc {\Npp} {N\llp t \rrp}
\nc {\Gbb} {G\llb t \rrb}
\nc {\cC} {\mathcal{C}}
\nc {\cB}{\mathcal{B}}
\nc{\ocB}{\bar{\cB}}
\nc{\ocA}{\bar{\cA}}
\nc {\bR}{\mathbb{R}}
\nc {\bQ} {\mathbb{Q}}
\nc {\cL}{\mathcal{L}}
\nc {\bI}{\bar{I}}

\nc{\tg} {\mathtt{g}}
\nc {\tc}{\mathtt{c}}

\nc {\oG} {\overline{G}}
\nc{\ofg}{\overline{\fg}} 

\nc{\Fq} {\mathbb{F}_q}
\nc{\Fqt}{\Fq\llp t \rrp}
\nc{\ocK}{\overline{\cK}}
\nc{\Gal}{\mathrm{Gal}}
\nc{\uG}{\underline{G}}
\nc {\fz}{\frak{z}}
\nc{\cW}{\mathcal{W}}
\nc{\tu}{\tilde{u}}
\nc{\gl}{\frak{gl}}
\nc {\hgl}{\widehat{\gl}}
\nc {\cdet}{\on{cdet}}
\nc {\Comp}{\on{Comp}}
\nc {\Cond}{\on{Cond}}
\nc {\Irr}{\on{Irr}}
\newcommand{\quash}[1]{}  
\nc {\triv}{\on{triv}}

\begin{document}

\title{On the notion of conductor in the local geometric Langlands correspondence} 

\author{Masoud Kamgarpour}

\thanks{The author was supported by an ARC DECRA Fellowship} 

\email{masoud@uq.edu.au}
\address{School of Mathematics and Physics, The University of Queensland, Australia}

\date{\today}



\begin{abstract}
Under the local Langlands correspondence, the conductor of an irreducible representation of $\Gl_n(F)$ is  greater than the Swan conductor of the corresponding Galois representation. In this paper, we establish the geometric analogue of this statement by showing that the conductor of a  categorical representation of the loop group is greater than the irregularity of the corresponding meromorphic connection. 
\end{abstract} 

\maketitle 

{\dedication{To Brian Forrest, for inspiring us to do math.}} 

\tableofcontents

\section{Introduction}   
\ssec{Arithmetic local Langlands correspondence} 
Let $F$ be a local non-Archimedean field such as  $\mathbb{Q}_p$ or $\mathbb{F}_q\llp t \rrp$. The local Langlands correspondence for $\GL_n$ relates two different  types of data: 
\[
\boxed{{\textrm{Representations of $\GL_n(F)$}}} \longleftrightarrow \boxed{\textrm{Representations of $\Gal(\bar{F}/F)$}}
\]
After appropriate modifications, the above relationship can be formulated as a bijection, and is now a theorem, cf. \cite{Wedhorn} for a review. This bijection preserves important numerical invariants associated to objects on the two sides. The invariant we consider in this paper is a positive integer known as the \emph{conductor}. 

\sssec{Preservation of conductor} 
 E. Artin defined the notion of conductor for Galois representations. This is a positive integer which, roughly speaking, measures how ramified the representation is.  On the other hand, the notion of conductor for irreducible representations of $\Gl_n(F)$ was defined in \cite{Cass} and \cite{JSPS}. We recall their definition in \S \ref{ss:conductor}. 
 
 It is known that under the local Langlands bijection
 \[ 
 \begin{matrix}
 \boxed{\textrm{conductor of a $\Gl_n(F)$-module}} &  =  & \boxed{\textrm{Artin conductor of the corresponding Galois module.}} 
 \end{matrix} 
 \]
 In fact, preservation of conductors played an important role in one of the proofs of the local Langlands correspondence, cf. \cite[\S 4.2.5]{Wedhorn}.  Our goal is to examine the analogue of this statement in the \emph{geometric} Langlands program. 
 
 An immediate problem is that I do not know what the geometric analogue of the Artin conductor is. This prevents us from discussing the geometric analogue of the above equality. Instead, we content ourselves with establishing a related inequality. 
 
 \sssec{Swan conductor} 
Artin conductor has a variant called the \emph{Swan conductor}. Roughly speaking, the Swan conductor measures how \emph{wildly} ramified the Galois representation is (see \cite{GR} for a review of Artin and Swan conductors). For our purposes, it is sufficient to know that the Artin conductor is greater than or equal to the Swan conductor. 
 
 We can, therefore, summarise the above discussion as follows: under the local Langlands correspondence
\[
\boxed{\textrm{Conductor of a $\Gl_n(F)$-module}} \geq \boxed{\textrm{Swan conductor of the corresponding Galois module.}}
\] 
In this paper, we establish the geometric analogue of this inequality.

\ssec{Dictionary for geometrisation} 
 To formulate the appropriate geometric statement, we use the following table of analogies:  
\[ 
  \begin{tabular}{|c|c|}
  \hline
  \textbf{Number theory} & \textbf{Geometry} \\\hline
  Galois representations & Meromorphic connections \\ \hline
  Swan conductor & Irregularity \\ \hline
   Representations of $\Gl_n(F)$ & Categorical representations of the loop group $\Gl_n\llp t\rrp$ \\ \hline
   Conductor of a $\Gl_n(F)$-module & Conductor of a categorical representation\\ \hline
  \end{tabular} 
\]

Meromorphic connections and their irregularities are classical topics, see, e.g. \cite{Deligne}, \cite{Katz}, \cite{Kedlaya}.  
We shall discuss categorical representations  in \S \ref{s:formulation} where we also recall Frenkel and Gaitsgory's formulation of the local geometric Langlands as a correspondence between the following two types of data: 
\[
\boxed{\textrm{Categorical representations of the loop group}} \longleftrightarrow \boxed{\textrm{Meromorphic connections on the disk}}
\]

\ssec{Rough version of our main result} In \S \ref{ss:conductor}, we define  the conductor of a categorical representation by a straightforward adaption of the definitions of \cite{Cass} and \cite{JSPS}. 
Here is a rough statement of our main theorem. For the precise statement, see \S \ref{ss:main}.
  
\begin{thm} 
 Under the local geometric Langlands correspondence
\[
\boxed{\begin{matrix} \textrm{conductor of a categorical } \\
\textrm{representation of $\Gl_n\llp t \rrp$} \,\,   \end{matrix}} \quad
\geq   \quad \boxed{\begin{matrix} \textrm{irregularity of the corresponding} \\ \textrm{rank $n$-meromorphic connection.}\, \,  \end{matrix}}
\]
\end{thm} 
As we shall see, the main ingredient of the proof is controlling the action of a specific set of Segal-Sugawara operators, defined by Chervov and Molev and Talalaev \cite{CM} on certain critical level representations of the affine Kac-Moody algebra $\hgl_n$.

\ssec{Concluding remarks} 
One can ask if there is a version of the above theorem for groups other than $\GL_n$. This is a question that already makes sense in the arithmetical setting; i.e., one may wonder if there is an analogue of conductor for smooth representations of a reductive $p$-adic group. We are unaware of such a notion in general, however, in \cite{Tsai}, a notion of new forms for odd special orthogonal groups has been defined. Thus, it appears one can define conductors for representations of such group over local fields. However, as far as we know, this has not been pursued further in the literature.

In general, we cannot expect the equality to hold in the above theorem. The best one can hope for is that the conductor of a categorical representation of $\GL_n(\!(t)\!)$ is less than or equal the irregularity of the corresponding connection plus $n$. For some applications of the conductor in local geometric Langlands correspondence, we refer the reader to \cite{Luu}.

 As mentioned above, the local Langlands correspondence is a theorem. However, I do not know of a straightforward construction which takes as input a Galois representation of $\Gal(\bar{F}/F)$ and produces a (smooth irreducible) representation of $\Gl_n(F)$. The situation in the geometric setting is, in some sense, reversed. There is a  construction of (what should be) the local Langlands correspondence (see \S \ref{ss:locGeom}), but the fact that this construction satisfies the correct properties remains highly conjectural.  
 
 Frenkel and Gaitsgory have extensively analysed the local geometric Langlands correspondence when the underlying connection is regular singular (i.e., its irregularity equals zero). Previously, we have examined some of the features of the irregular case \cite{Masoud}, \cite{CK}, \cite{MasoudTravis}. This paper is an attempt to understand yet another aspect of the local geometric Langlands correspondence in the presence of irregular connections.

\ssec{Acknowledgements} I would like to thank Tsao-Hsien Chen, Kiran Kedlaya, Martin Luu, Alex Molev, Arun Ram, Sam Raskin, and Claude Sabbah for helpful conversations.

\ssec{Organisation of the text} In \S \ref{s:formulation} we recall the definitions of meromorphic connections, categorical representations, and the local geometric Langlands correspondence. We then define the notion of conductor of categorical representations, give a precise version of our main theorem, and outline  the proof. 

In \S \ref{s:smoothRep} we study smooth representations of an arbitrary affine Kac-Moody algebra. We will recall the definition of the completed enveloping algebra and its relationship to the affine vertex algebra. In addition, we introduce a class of modules, called \emph{root modules}, and investigate the action of centre on these modules. The main result of this section is a vanishing result regarding the action of Fourier coefficient of Segal-Sugawara operators on root modules. 

In \S \ref{s:gln}, we specialise to the case of $\hgl_n$. We recall Chervov and Molev's explicit description of a complete set of Segal-Sugawara operators, and identify them as the ones arising from functions on opers. We then use the above-mentioned vanishing result to establish the main theorem.

\section{Formulation of the main result}\label{s:formulation}
In \S \ref{ss:connections}, we recall a few facts we need about meromorphic connections, including the cyclic vector theorem and the Komatsu-Malgrange formula for irregularity. In \S \ref{ss:catRep}, we will discuss the main class of categorical representations which we consider in this article. In \S \ref{ss:locGeom}, we  recall the  Feigin-Frenkel theorem and Frenkel and Gaitsgory's version of the local geometric Langlands \cite{FG}. In \S \ref{ss:conductor}, we define the notion of conductor for a categorical representation of $\Gl_n\llp t\rrp$. Armed with this information, we give a precise version of our result in \S \ref{ss:main} and give a sketch of the proof in \S \ref{ss:sketch}.

\subsection{Meromorphic connections}\label{ss:connections}  Let $\cK=\bCpp$ and  let $\cDt=\Spec(\cK)$  be the punctured disk. Let $V$ be a finite dimensional vector space over $\cK$. A \emph{differential operator} on $V$ is a $\bC$-linear map $D: V\ra V$ satisfying 
\[
D(av)=(\partial_t a)v + aD(v),\quad \quad  a=a(t)\in \cK, \quad v\in V.
\]
A \emph{connection} on $\cDt$ is a pair $(V,D)$ consisting of a vector space $V$ over $\cK$ together with a differential operator $D:V\ra V$.  We say that $\nabla=(V,D)$ has rank $n$ if $\dim(V)=n$.  We denote by $\Conn_n(\cDt)$  the set of all rank $n$ connections on $\cDt$. 

\sssec{Cyclic vector theorem} Let $\nabla=(V,D)$ be a connection on $\cDt$. A \emph{cyclic vector} for $\nabla$ is a vector $v\in V$ such that 
\[
\{v,D.v\dots, D^{n-1}v\}
\]
 is a basis for $V$. In this case, the differential operator $D$ is completely determined by the $n$-tuple $(a_1,\dots, a_n)\in \cK^n$ defined by the equation
\begin{equation}\label{eq:ntuple}
D^nv=a_{1}D^{n-1}v+\cdots + a_{n-1}Dv+ a_{n}v.
\end{equation}
 Note, however, that such $n$-tuple is not unique; it depends on the choice of the cyclic vector.  
 
 According to a theorem of Deligne (cf. \cite[p42]{Deligne} and  \cite[\S 5.6]{Kedlaya}), every connection in $\Conn_n(\cDt)$ has a cyclic vector.

\sssec{Opers} An \emph{oper} on $\cDt$ is a triple $(V,D,v)$ where $v$ is a cyclic vector for the connection $(V,D)\in\Conn_n(\cDt)$. We note that opers for general reductive groups $G$ were defined by Beilinson and Drinfeld \cite{BDOper} following earlier work of Drinfeld and Sokolov \cite{DS}.  In the case of $G=\Gl_n$, after a slight modification (cf. \cite{GaitsgorySeminar}), their definition becomes equivalent to the one given above, cf. \cite[\S 16.1]{FBz}.

Let $\Op_n(\cDt)$ denote the set of all opers $(V,D,v)$ with $\dim(V)=n$. By the above discussion, specifying an oper amounts to specifying an $n$-tuple $(a_1,\dots, a_n)\in \cK^n$; thus, we have an isomorphism
\[
\Op_n(\cDt) \simeq \cK^n. 
\]
Therefore, we see that $\Op_n(\cDt)$ has a very simple description whereas $\Conn_n(\cDt)$ is complicated. The two spaces are, however, intimately related. 
Namely, we have a canonical forgetful map
\[
\begin{matrix}
p: & \Op_n(\cDt) & \longrightarrow & \Conn_n(\cDt)\\
 &  (V,D,v) & \longmapsto & (V,D).
\end{matrix} 
\]
The cyclic vector theorem states that this map is surjective. Note that the map $p$ is by  no means injective. It is known that the geometry of the fibres is related to the affine Springer fibres, \cite{FZ}. 

\sssec{Irregularity} 
To a connection $\nabla \in \Conn_n(\cDt)$ one can associate a non-negative integer $\Irr(\nabla)$ called the \emph{irregularity}. This integer measures, in some sense, how \emph{wildly singular} the connection $\nabla$ is. 

There are several ways for defining $\Irr(\nabla)$. For instance, one can define it as the sum of slopes of the connection $\nabla$ (cf. \cite[\S 2.3]{Katz}). For us, it will be convenient to define this invariant using cyclic vectors, or equivalently, using opers. 

\begin{defe} The  \emph{irregularity} of an oper $\chi=(a_1,\dots, a_n)\in \Op_n(\cDt)\simeq \cK^n$ is defined by
\begin{equation}\label{eq:Irr} 
\boxed{\Irr(\chi):=\max \{i-\nu(a_{n-i})\}_{i=0,\dots,n-1} - n.}
\end{equation} 
Here $-\nu(a_{n-i})$ denotes the order of pole of the Laurent series $a_{n-i}\in \cK$. 
\end{defe} 
One can show that the irregularity of an oper depends only on the underlying connection.  In other words, 
\begin{equation} \Irr(\chi)=\Irr(\nabla), \quad \quad \textrm{where $\nabla=p(\chi)$}.
\end{equation} 
See, e.g., \cite[\S 7]{Kedlaya}. We note that \eqref{eq:Irr} is sometimes known as the \emph{Komatsu-Malgrange formula} (cf. \cite{Komatsu}, \cite{Malgrange}).\footnote{I thank Claude Sabbah for bringing this equality to my attention.}

\sssec{Relationship to Galois representations} It has been known for a long time that connections on the punctured disk behave very similarly to finite dimensional representations of $\Gal(\mathbb{Q}_p)$ and $\Gal(\Fq\llp t \rrp)$, cf. \cite{Katz}. As far as I know, there is no formal mathematical theory embodying both worlds. In addition, there is no analogue of opers in the arithmetic world. The existence of opers is one of the key simplifying ingredients in the geometric Langlands program.

\subsection{Categorical representations} \label{ss:catRep} Having discussed the geometric analogue of Galois representations, we now turn our attention to the geometric analogue of representations of $\Gl_n(F)$. According to \cite{FG}, these should be certain categorical representations of the loop group $\Gpp$. As a toy model, 
 let $G$ be an algebraic group and $\fg$ denote its Lie algebra. Then $G$ acts on the category $\fg-\modd$ by auto-equivalences:  $g\in G$ sends a representation $\fg\ra \End(V)$ to a new representation: the one defined by the composition
\[
\fg\rar{\mathrm{Ad(g)}} \fg \ra \End(V).
\]
 This is an example of a categorical action. It is also possible to ``decompose'' this categorical representation using the centre  of the universal enveloping algebra. Namely, for every character $\chi$ of the centre, let $\fg-\modd_\chi $ denote the full subcategory of $\fg-\modd$ consisting of those modules on which $\cZ(\fg)$ acts by the character $\chi$. Then $\fg-\modd_\chi$ is preserved under the action of $G$; thus, it is a sub-representation of $\fg-\modd$. We will be interested in the analogous categorical representations for the \emph{loop group}. 

So let  $\fgpp:=\fg\otimes \bCpp$ denote the \emph{loop algebra}. The corresponding group $\Gpp$ is the \emph{loop group} associated to $G$. It is known that $\Gpp$ has a structure of an ind-scheme, though this is used only implicitly in this article.   
The categorical representations we study arise from the action $\Gpp$ on the category of $\fgpp$-modules. 
Actually, in the case $\fg$ is reductive, it is fruitful to consider not the loop algebra itself, but its universal central extension known as the \emph{affine Kac-Moody algebra} $\hfg$.

 Recall that representations of $\hfg$ have a parameter, a complex number $k$, which is called the \emph{level} \cite{Kac}. We let $\hfg_k-\modd$ denote the category of (smooth) representations of $\hfg$ at level $k$. The adjoint action of the loop group $\Gpp$ on $\hfg$ preserves the central line; thus, $\Gpp$ acts on the category $\hfg_k-\modd$. The centre of the (completed) universal enveloping algebra is nontrivial only when $k$ is a specific complex number called the \emph{critical level}.\footnote{In the normalisation of \cite{Kac}, the critical level for $\widehat{\mathfrak{sl}}_n$ is $k=-n$.} Thus the procedure of decomposing a categorical representation according to central characters can only be carried out at the critical level. 
 
 Let $\cZ_c$ denote the centre of the completed universal enveloping algebra of the affine Kac-Moody algebra $\hfg$ at the critical level. 
 Now every point $\chi\in \Spec(\cZ_c)$ defines a character of the centre. Let $\hfg_c$ denote the category of (smooth) representations at the critical level, and let $\hfg_c-\modd_\chi\subset \hfg_c-\modd$ denote the full subcategory consisting of those representations on which the centre acts according to the character $\chi$. It is easy to see that the action of $\Gpp$ on $\hfg_{c}-\modd$ preserves $\hfg_c-\modd_\chi$. 
 
 Frenkel and Gaitsgory propose that the categorical representations $\hfg_c-\modd_\chi$ should be the geometric analogue of representations of $\Gl_n(F)$. We refer the reader to the introduction of \cite{FG} for a discussion about the motivations for this proposal. The relationship between these categorical representations and meromorphic connections emerges through the Feigin-Frenkel Theorem.

  \subsection{Feigin and Frenkel's theorem and  geometric Langlands}\label{ss:locGeom} Henceforth, we consider the case $\fg=\gl_n$. 
   According to a fundamental theorem of Feigin and Frenkel \cite{FF}, \cite{FrenkelBook}, we have an isomorphism of commutative (topological) algebras 
\begin{equation}\label{eq:FF}
 \cZ_c \simeq \bC[\Op_n(\cDt)].
 \end{equation} 
 
 Frenkel and Gaitsgory \cite{FG} formulate the local geometric Langlands program as a correspondence 
 \[
 \begin{matrix} 
 \boxed{\textrm{$\Conn_n(\cDt)$}} & \longleftrightarrow & \boxed{\textrm{Categorical representations of $\Gl_n\llp t \rrp$}} \\\\
 \nabla=p(\chi) & \longleftrightarrow & (\hgl_n)_c-\modd_\chi.\\\\
 \end{matrix} 
 \]
 In more details, given an oper $\chi\in \Op_n(\cDt)$, one has a  character of the centre $\cZ_c$ and, therefore, a categorical representation $(\hgl_n)_c-\modd_\chi$. Frenkel and Gaitsgory propose that this is the categorical representation corresponding to the connection $p(\chi)$ (the underlying connection of the oper $\chi$).  
 
  For this correspondence to be well-defined, one must have that $\hfg_c-\modd_\chi$ depends only on the connection $\nabla$; i.e., it is independent of the chosen cyclic vector. Frenkel and Gaitsgory conjecture that this is indeed the case.

\subsection{Conductor of categorical representations} \label{ss:conductor} We now discuss how to define the conductor of  categorical representations of the loop group $\Gpp$. 
To this end, we first recall the definition of the conductor of an irreducible smooth representation of $\Gl_n(F)$, cf. \cite[\S 2.5.4]{Wedhorn}.
 Let $\cO$ denote the ring of integer of the non-Archimedean local field $F$ and let $\cP$ denote its maximal ideal. For every nonnegative integer $m$, let $\fk_m$ denote the Lie subalgebra of $\gl_n(\cO)$ defined by 
\begin{equation}\label{eq:pm}
\fk_m:=
\begin{bmatrix} 
\cO & \cdots & \cO  \\
\vdots & \ddots  &  \vdots \\
\cO & \cdots & \cO  \\
\cP^m & \cdots &  \cP^m
\end{bmatrix}.
\end{equation}
Let $K_m$ denote the corresponding subgroup of $\Gl_n(F)$. Let $V$ be a (smooth irreducible) $\Gl_n(F)$-module. Then 
\[
\textrm{Conductor of $V$ := smallest nonnegative  integer $m$ such that $V$ has a $K_m$-invariant vector.}
\] 

We now move to the geometric theory, so we replace $F$ by $\cK=\bCpp$ and use $\cO$ (resp. $\cP$) to denote the ring of integers of $\cK$ (resp. its maximal ideal). Let $\fk_m$ be the subalgebra of the loop algebra $\fgpp$ defined as in \eqref{eq:pm}. Let $K_m$ denote the corresponding subgroup of the loop group $\Gpp$. It is known that $K_m$ has a structure of a pro-algebraic group over $\bC$. 

In categorical representation theory, the notion of ``equivariant objects'' replaces the concept of ``invariant vectors'' \cite[\S 20]{FG}, \cite[\S 10]{FrenkelBook}. Thus, to define the notion of conductor for categorical representations, we should consider $K_m$-equivariant objects. 

Before giving the definition, let us recall a general fact. Let $\fk$ be a of $\fg(\cO)$ and let $K$ denote the corresponding subgroup of the loop group. Suppose $K$ has a structure of a pro-algebraic group over $\bC$. It turns out that for the categorical representations $\hfg_c-\modd_\chi$, an object $V$ is $K$-equivariant if and only if $V$ can be endowed with the structure of a $(\hfg_c, \fk)$ Harish-Chandra module, cf. \cite[\S 10]{FrenkelBook}. This means that $V$ is a $\hfg_c$-module such that the $\fk$-module obtained by restriction is integrable. Thus, we can define the notion of conductor as follows:

\begin{defe} The conductor of a categorical representation $\hfg_c-\modd_\chi$, denoted by $\Cond(\hfg_c-\modd_\chi)$, is the smallest nonnegative integer $m$ such that $\hfg_c-\modd_\chi$ contains a $(\hfg_c,\fk_m)$ Harish-Chandra module. 
\end{defe} 

We note that it is not obvious that such an integer $m$ exists. In the arithmetical setting (i.e. over a non-Archimedean field), this follows from the fact that every smooth irreducible representation of $\Gl_n(F)$ is \emph{admissible}. It would be interesting to understand the geometric analogue of this statement.

\subsection{Main Theorem} \label{ss:main} Recall that associated to an oper $\chi\in \Op_n(\cDt)$, we have the corresponding connection $\nabla=p(\chi)\in \Conn_n(\cDt)$ and the corresponding categorical representation $\hfg_c-\modd_\chi$. 

\begin{thm}\label{t:main} For all $\chi \in \Op_n(\cDt)$, we have 
\begin{equation} 
\Cond(\hfg_c-\modd_\chi)\geq  \Irr(\nabla).
\end{equation} 
\end{thm} 

The inequality in the theorem can be strict.\footnote{From the arithmetic perspective, this is not a surprise since the irregularity is the geometric analogue of the Swan conductor which is, in general, smaller than the Artin conductor.}   For instance, one can show (cf. \cite[\S 10.3.2]{FrenkelBook}) that
\[
\textrm{conductor of $\hfg_c-\modd_\chi$ is zero} \quad \iff\quad  \textrm{$\nabla$ is trivial}.
\]
 In particular, if $\nabla$ is a regular singular but nontrivial connection on $\cDt$, then irregularity of $\nabla$ is zero while the conductor of $\hfg_c-\modd_\chi$ is greater than zero.

\subsection{Outline of the proof}\label{ss:sketch} We now give an outline of the proof and the structure of the what comes next. 

  As we shall see in \S \ref{ss:proof}, 
 it is easy to reduce the problem to a representation-theoretic statement. Namely, let $\fk_m^0$ denote the pro-nilpotent radical of  $\fk_m$. Let $V\in \hfg_c-\modd_\chi$. Then Theorem \ref{t:main} reduces to the following statement: 

\begin{equation} \label{eq:reptheory}
\textrm{If $V$ contains a vector $v$ with $\fk_m^0.v=0$, then $\Irr(\chi)\leq m$}. 
\end{equation} 

Now we know that the oper $\chi$ can be represented by an $n$-tuple 
\[
(a_1,\dots, a_n),\quad \quad a_i=\sum_{k\in \bZ} a_{i,k} t^{-k-1}\in \cK.
\] 
To prove the above statement, we need to control the singularities of each $a_i$. 

To be more precise, for $i\in \{1,2,\dots, n\}$ and $k\in \bZ$ define functions
\[
\begin{matrix} 
 v_{i,k} & : & \Op_n(\cDt) & \longrightarrow &  \bC, \\
 \chi & = & (a_1,\dots, a_n) & \longmapsto & a_{i,k}.\\
 \end{matrix} 
 \]
  Let $S_{i,k}\in \cZ_c$ denote the corresponding elements of $\cZ_c$. Using the Komatsu-Malgrange Formula, the statement \eqref{eq:reptheory} reduces to the following: 
\begin{equation}\label{eq:vanish}
S_{i,r}.v=0, \quad \quad r\geq m+i-1, \quad \quad \forall \, i\in \{1,2,\dots, n\}. 
\end{equation} 

The following observation will play a key role in proving this vanishing statement: 
\begin{equation} \label{eq:CMDescribe}
\textrm{The $S_{i,r}$'s are essentially the Segal-Sugawara operators defined in \cite{CM}.}
\end{equation} 
In other words, the central elements $S_{i,r}$ which are defined, rather abstractly,   via coefficients of opers, have in fact a very explicit description provided by Chervov and Molev and Talalaev. This key property of these operators is a consequence of the description of their image under the affine Harish-Chandra homomorphism.  Section \S \ref{s:gln} is devoted to proving the statements  \eqref{eq:CMDescribe} and \eqref{eq:vanish}.

With the explicit description of $S_{i,r}$'s at hand, proving Equation \ref{eq:vanish} then follows from the properties of the Fourier coefficients of the vertex operators in the affine vertex algebras and their action on smooth modules.  The relevant properties of these vertex operators are established in \S \ref{s:smoothRep}.



\section{Recollections on smooth representations} \label{s:smoothRep}

In this section, we study the action of the centre of the (completed) universal enveloping algebra of an (arbitrary) affine Kac-Moody algebra on certain smooth modules, which we call root modules. To this end, we recall the definition of the affine vertex algebra and Segal Sugawara operators. We then prove a vanishing result about the action of Segal Sugawara operators on critical level root modules. This result, which also appeared in \cite{CK}, will be further refined in \S \ref{s:gln} in the case of $\hfg=\hgl_n$.

 \subsection{Smooth modules for affine Kac-Moody algebras}\label{ss:smoothRep} 
 \sssec{Loop space and arc space} 
 For a vector space $V$, we write 
 \[
 \begin{matrix} V\llb t \rrb & \textrm{for the (formal) arc space} & V\otimes \bCbb, \\
V\llp t \rrp &  \textrm{for the (formal) loop space}  & V\otimes \bCpp.
\end{matrix} 
\]
 If $V=\fg$ is a Lie algebra, then $\fg\llb t \rrb$ and $\fg\llp t \rrp$ are also Lie algebras with bracket
 \[
[x\otimes t^m, y\otimes t^n] = [x,y]\otimes t^{m+n}, \quad \quad x,y\in \fg. 
\]

\sssec{Affine Kac-Moody algebras} 
Let $\fg$ be a simple finite dimensional Lie algebra. Let $\kappa$ be an invariant non-degenerate bilinear form on $\fg$. The affine Kac-Moody algebra $\hfg_\kappa$  is defined to be the central extension
\begin{equation}\label{eq:central}
0\ra \bC.\bone \ra \hfg_\kappa\ra \fgpp \ra 0,
\end{equation} 
with the two-cocycle defined by the formula 
\begin{equation}\label{eq:2cocycle}
(x\otimes f(t), y\otimes g(t)) \mapsto -\kappa(x,y).\,\Res_{t=0} \left(f(t)\frac{d}{dt}g(t)\right).
\end{equation} 
Here, $\frac{d}{dt}g(t)$ denotes the derivative $g(t)\in \fgpp$ and $\Res_{t=0}$ denotes the coefficient of $t^{-1}$.\\

A $\hfg_\kappa$-module is a vector space equipped with an action of $\hfg_\kappa$ such that the central element $\bone\in \bC.\bone \subset \hfg_\kappa$ acts as identity. Let  
\[
U_\kappa(\hfg):=U(\hfg_\kappa)/(\bone-1)
\]
 be the quotient of the universal enveloping algebra by the ideal generated by $\bone-1$. Then $U_\kappa(\hfg)$ is the universal enveloping algebra at level $\kappa$ and there is an equivalence of  categories  
 \[
\begin{matrix} 
U_\kappa(\hfg)-\modd  & \simeq & \hfg_\kappa-\textrm{modules}
 \end{matrix} 
 \]

\sssec{Smooth modules} 
A module $V$ over $\hfg_\kappa$ is \emph{smooth} if for every $v\in V$ there exits a positive integer $N_v$ such that 
\[
t^{N_v} \fgbb.v=0.
\]
In \cite{Kac}, these are called \emph{restricted} modules. 

 Let $\hfg_\kappa-\modd$ denote the category of \emph{smooth} $\hfg_\kappa$ modules. Let $\tU_\kappa(\hfg)$ denote the \emph{completed} universal enveloping algebra at level $\kappa$. By definition, 
 \[
 \tU_\kappa(\hfg)=\varprojlim U_\kappa(\hfg)/I_N,
 \]
 where $I_N$ is the left ideal generated by $t^N \fgbb$. Then, there is an equivalence of categories 
  \[
 \tU_\kappa(\hfg)-\modd\simeq \hfg_\kappa-\modd.
 \] 

Let $\cZ_\kappa(\hfg)$ denote the centre of $\tU_\kappa(\hfg)$. One can show that $\cZ_\kappa$ is trivial for all but one value of $\kappa$. When $\kappa$ is the \emph{critical} level, that is $\kappa$ equals $-\frac{1}{2}$ times the Killing form, then the centre $\cZ_c(\hfg)$ is very interesting. In what follows, we define certain smooth modules and study the action of $\cZ_c(\hfg)$ on them.

 \subsection{Root subalgebras and modules} \label{ss:root} 
  \subsubsection{Root space decomposition} 
  Let $\fh$ be a Cartan subalgebra of $\fg$ and $\Phi$ be the corresponding root system. For ease of notation, we set
\[
\Phi^*:=\Phi \sqcup \{0\}.
\]
 We have then the root decomposition 
\[
\fg=\bigoplus_{\alpha\in \Phi^*} \fg^\alpha,
\]
where $\fg^\alpha$ is the eigenspace for root $\alpha\in \Phi$ and  $\fg^0=\fh$. 

In what follows, we write $x^\alpha$ for an element of $\fg^\alpha$ and  $x_n^\alpha$ for the element $x^\alpha \otimes t^n\in \fg^\alpha\llp t \rrp$. Following the usual abuse of notation, we also  write $x_n^\alpha$ for the corresponding element of $\tU_\kappa(\hfg)$.

\subsubsection{Root functions}
Now let 
\[
r:\Phi^*\ra \bZ
\]
 be a function. Define a subspace 
\[
\fk=\fk_r:=\bigoplus_{\alpha\in \Phi^*} t^{r(\alpha)} \fg^\alpha\llb t \rrb\subset \fgpp.
\]
 We assume that $r$ satisfies the following properties 
 \begin{equation} \label{eq:assumption0}
 \begin{matrix} 
 r(\alpha)\geq 0 & &\forall\, \alpha \\
 r(\alpha)+r(\beta)\geq r(\alpha+\beta) & & \forall \,\alpha, \beta.
\end{matrix} 
\end{equation} 
This ensures that $\fk_r$ is a subalgebra of $\hfg$ (in fact a subalgebra of $\fg[\![t]\!]$) and the central extension is split over $\fk_r$. For technical reasons, we also assume that 
$r(0)>0$. (These assumptions are satisfied in the case of interest to us). 

\sssec{Root modules} 
Define
 \begin{equation}\label{eq:induced}
 V_\fk:=\Ind_{\fk\oplus \bC.\bone}^{\hfg_\kappa}(\bC).
 \end{equation} 
 We call $V_\fk$ the \emph{root module} associated to the root subalgebra $\fk$. 
 It is clear  this module is smooth; i.e., $V_\fk\in \hfg_\kappa-\modd$.  One way to define this module is to say that $V_\fk$ is generated by a vector $v_0$ subject to the relations 
 \begin{equation}\label{eq:relations}
 \begin{matrix} 
 \bone.v_0 & = & v_0  & \\
 x_n^\alpha.v_0 & = & 0 &   \forall \alpha\in \Phi^*, \, \, \forall \, n\geq r(\alpha). 
 \end{matrix} 
 \end{equation}
 
 \begin{lem} \label{l:relations}
 Let $x=x_{n_1}^{\alpha_1} x_{n_2}^{\alpha_2}\dots x_{n_k}^{\alpha_k} \in \tU_\kappa(\hfg)$. 
 Suppose $\displaystyle \sum_{i=1}^k n_i \geq \sum_{i=1}^k r(\alpha_i)$. Then $x.v_0=0$. 
 \end{lem}
 
 \begin{proof} This is established by induction using commutation relations in $\hfg_\kappa$. We refer the reader to \cite[\S 3.3]{CK} for the proof. 
 \end{proof} 
 
Our aim is to study the action of $\cZ_c(\hfg)$ on root modules. To obtain a description of $\cZ_c(\hfg)$, we need to take a detour through the theory of vertex algebras. We shall see that the above lemma allows us to obtain estimates on the action of the centre on root modules.

 \subsection{Affine vertex algebras} \label{ss:vertex}
 \sssec{Recollections on vertex algebras} 
 A vertex algebra is a vector space equipped with a triple $(Y, T, \bv)$ where $\bv\in V$ is a fixed vector, called the \emph{vacuum}, $T:V\ra V$ is a linear operator, called the \emph{translation}, and $Y$ is a  map 
  \[
 Y: V\ra \End\, V \llb z,z^{-1}\rrb.
 \]
 These data must satisfy some axioms, cf. \cite{FBz}.
Elements of $V$ are called \emph{states} and those of $\End\, V \llb z,z^{-1}\rrb$ are called \emph{fields}. Thus, $Y$ is known as  the \emph{states-fields correspondence}.   
 For $A\in V$, we write 
 \[
 Y(A)=\sum_{n\in \bZ} A_{(n)} z^{-n-1}. 
 \]
The operators $A_{(n)}\in \End \, V$ are known as the \emph{Fourier coefficients} of $A$. We mention some properties of the Fourier coefficients. Depending on the approach, some of these are taken as a part of the defining axioms. 

First of all, for $A,B\in V$, we must have $A_{(n)}B=0$ for $n>\!\!>0$. Furthermore, we must have the following equalities in $\End\, V$
\begin{equation} \label{eq:commutator}
[A_n, B_m]=\sum_{n\geq 0} {m \choose n} (A_{(n)}B)_{m+n-k}.
\end{equation} 
and 
\begin{equation} \label{eq:translation}
(TA)_{(n)}=-nA_{(n-1)}.
\end{equation} 

\sssec{Lie algebra of Fourier coefficients} 
Identity \eqref{eq:commutator} implies that the span in $\End \, V$ of all the Fourier coefficient  is a Lie subalgebra.  In fact, more is true, cf. \cite[\S 4]{FBz}. Namely, one has a Lie algebra structure on the ``abstract vector space'' $U'(V)$ spanned by the Fourier coefficients $A_{(n)}$, $n\in \bZ$, subject to the relation \eqref{eq:translation}. 
We let $A_{[n]}$ denote the image of $A_{(n)}$ in $U'(V)$. By an abuse of notation, we also call $A_{[n]}$ the $n^\mathrm{th}$ Fourier coefficient of $A$. We will also consider the Lie algebra $U(V)$ spanned by infinite linear combinations of $A_{[n]}$'s subject to the relation \eqref{eq:translation}. Thus, $U(V)$ is a certain completion $U'(V)$.

 \sssec{Affine vertex algebra} 
 In this paper, we will be concerned with the vertex algebra associated to the affine Kac-Moody algebra $\hfg_\kappa$.
Let
 \[
 V_\kappa(\hfg):=\Ind_{\fgbb\oplus \bC.\bone}^{\hfg_\kappa}(\bC). 
 \]
 Then $V_\kappa(\hfg)$ carries the structure of a vertex algebra, known as the \emph{affine vertex algebra at level $\kappa$}.  Note $V_\kappa(\hfg)$ is generated by a vector $v_0$ subject to the relation 
 \[
 \fgbb.v_0=0.
 \]
  The operator $T$ is specified by the following requirements
 \[
 T(v_0)=0, \quad \quad [T,x_n]=-nx_{n-1}, \quad x\in \fg.
 \]
   Note that the vector space $V_\kappa(\hfg)$ is spanned by elements of the form $x_{n_1}^{\alpha_1}\dots x_{n_k}^{\alpha_k}$ where $n_i<0$.
 The Fourier coefficients of  $x_{-1}^\alpha$ are easy to describe; namely,  
 \[
(x_{-1}^\alpha)_{(s)}:= x_{s}^\alpha,\quad \quad \forall s\in \bZ.
 \]
In fact, the Reconstruction Theorem (cf. \cite{FBz}) guarantees that the above formula determines the fields associated to \emph{every} state in our vector space.  For our purposes, it  suffices to know that for every $n<0$, the Fourier coefficient $(x_{n}^\alpha)_{(s)}$ is a constant multiple of  $x_{s+n+1}^{\alpha}$. More generally,  let $X=x_{n_1}^{\alpha_1}\dots x_{n_k}^{\alpha_k}$. Then, $X_{(s)}$ is a linear combination of elements of the form 
\begin{equation}\label{eq:Fourier} 
(x_{n_{\sigma(1)}}^{\alpha_{\sigma(1)}})_{(s_1)} \dots (x_{n_{\sigma(k)}}^{\alpha_{\sigma(k)}})_{(s_k)}=
x_{n_{\sigma(1)}+s_1+1}^{\alpha_{\sigma(1)}} \dots x_{n_{\sigma(k)}+s_k+1}^{\alpha_{\sigma(k)}}
\end{equation} 
 where $\sigma$ ranges over automorphisms of $\{1,2,\dots, k\}$ and $s_i$'s are integers satisfying $\displaystyle \sum_{i=1} ^k s_i = s$. Thus, we also have a similar description of $X_{[s]}$ in terms of the Fourier coefficients $(x_{n}^\alpha)_{[r]}$.

 \sssec{Action of Fourier coefficients} 
One can show that  the map $(x_{-1}^\alpha)_{[r]} \ra x_r^\alpha$ extends to an \emph{injective Lie algebra homomorphism} 
\begin{equation}\label{eq:embedding} 
U(V_\kappa(\hfg))\inj \tU_\kappa(\hfg).
\end{equation}
In particular, we have an action of Fourier coefficients $X_{[n]}\in U(V_\kappa(\hfg))$ on smooth $\hfg_\kappa$-modules. The key fact we need is the following vanishing result.

 \begin{prop} \label{p:estimate} 
 Let $r$ be a function $r:\Phi^*\ra \bZ$ satisfying \eqref{eq:assumption0} and $r(0)>0$. Let $v_0$ be a vector satisfying   \eqref{eq:relations}. Then 
 \[
 (x_{n_1}^{\alpha_1}\dots x_{n_k}^{\alpha_k})_{[s]}.v_0=0,\quad \quad \displaystyle s\geq \sum_{i=1}^k r(\alpha_i) -\sum_{i=1}^k n_i-k.
 \]
 \end{prop} 
 
 \begin{proof} The proof is an easy application of Lemma \ref{l:relations}.
 Suppose $s_i$'s are integers satisfying $\displaystyle \sum_{i=1} ^k s_i = s$. Then the assumption on $s$ implies that 
 \[
\displaystyle \sum_{i=1} ^k (n_i+ s_i+1) \geq \sum_{i=1}^k r(\alpha_i).
\]
 By Lemma \ref{l:relations}, every element of the form \eqref{eq:Fourier} annihilates $v_0$. Now the image of 
 $ (x_{n_1}^{\alpha_1}\dots x_{n_k}^{\alpha_k})_{[s]}$ under the morphism \eqref{eq:embedding} equals $(x_{n_1}^{\alpha_1}\dots x_{n_k}^{\alpha_k})_{(s)}$ which, in turn, is a linear combination of elements of the form \eqref{eq:Fourier}. As each of these terms annihilate $v_0$, so does $ (x_{n_1}^{\alpha_1}\dots x_{n_k}^{\alpha_k})_{[s]}$.
  \end{proof}

 \subsection{Centre of the affine vertex algebra} \label{ss:centreVertex}
 \sssec{Centre of vertex algebras} 
 Having discussed the basic structure of the affine vertex algebra, we are ready to study its centre. 
 The \emph{centre} of a vertex algebra $V$ is the commutative vertex subalgebra spanned by $B\in V$ such that $A_{(n)}.B=0$ for all $n\geq 0$ and all $A\in V$. The following simple observation will play an important role: if $B$ is in the centre of $V$, then Identity \eqref{eq:commutator} implies that $B_{[n]}$  is in the centre of $U'(V)$ and the centre of $U(V)$.

 Let $\fz_\kappa$ denote the centre of the vertex algebra $V_\kappa(\hfg)$. Our goal is to describe $\fz_\kappa$ and use this to shed light on the centre $\cZ_\kappa$ of the completed universal enveloping. As mentioned above, the centre is only interesting when $\kappa$ is critical. Henceforth, we denote the centre of the affine vertex algebra at the critical level by $\fz_c=\fz_c(\hfg)$.

 \sssec{Segal-Sugawara operators} 
Note that as a vector space $V_\kappa(\hfg)$ is isomorphic to the universal enveloping algebra $U(\fg_-)$, where 
\[
\fg_-:=\fg\otimes t^{-1}\bC\llb t^{-1}\rrb.
\]
 Given $S\in U(\fg_-)$, we write $\bar{S}$ for its image in the associated graded algebra 
 \[
 \gr(U(\fg_-)) \simeq S(\fg_-).
 \]
   Note that we have an embedding $\fg\inj \fg_-$ given by $x\mapsto x_{-1}=x\otimes t^{-1}$ which induces an embedding $S(\fg)\inj S(\fg_-)$.  The following definition is due to Chervov and Molev \cite{CM} following \cite{CT}. 

\begin{defe}\label{d:CM}
 A complete set of Segal-Sugawara vectors is a set of elements 
\[
S_1, S_2, \dots, S_{n} \in U(\fg_-), \quad n=\mathrm{rk} \fg,
\]
where $S_i\in \fz_c$ and $\bar{S_1}, \dots, \bar{S_n}$ coincide with the images of some algebraically independent generators of the algebra of invariants $S(\fg)^{\fg}$ under the imbedding $S(\fg)\inj S(\fg_-)$.
\end{defe} 
  
 Note that the elements $S_i$ are by no means unique. This is related to the fact  that there are many choices for generators of the   polynomials algebra $S(\fg)^\fg$.

  \sssec{Feigin-Frenkel Theorem} 
According to the Feigin-Frenkel Theorem \cite{FrenkelBook}, $\fz_c$ is a polynomial algebra in infinitely many variables; more precisely, if $S_1,\dots, S_n$ is a complete set of Segal-Sugawara operators, then  
\begin{equation} \label{eq:FFA1}
\fz_c \simeq \bC[T^r S_i \, | \, i=1, \dots, n, \, \, r\geq 0].
\end{equation}

\subsection{Harish-Chandra homomorphism} \label{s:HC}
Let $\fg$ be a simple finite dimensional Lie algebra. Consider the natural projection $U(\fg)\ra U(\fh)$. This is merely a morphism of vector spaces. Taking $\fh$-invariants, however, gives rise to a commutative diagram of algebra homomorphisms 
\[
\xymatrix{
U(\fg)^\fh \ar[r] &U(\fh)=\Sym(\fh)\\
\cZ(\fg) \ar@{^{(}->}[u] \ar[r]^\cong & \Sym(\fh)^W\ar@{^{(}->}[u].
}
\]
Here, $\Sym(\fh)^W$ denotes the space of  elements of $\Sym(\fh)$ invariant under the \emph{dotted} action of $W$.

We have a similar picture for affine vertex algebras.\footnote{I thank Alex Molev for explaining this to me.} Recall the notation  
\[
\fg^-:=t^{-1}\fg[t^{-1}],\quad \quad \fh^-:=t^{-1}\fh[t^{-1}].
\]
 Then, the canonical projection $U(\fg^-)\ra U(\fh^-)$ gives rise to a commutative diagram 
\begin{equation}\label{eq:commutative}
\xymatrix{
U(\fg^-)^\fh \ar[r] &U(\fh^-)=\Sym(\fh^-)\\
\fz_c \ar@{^{(}->}[u] \ar[r]^\cong & W(\cfg) \ar@{^{(}->}[u].
}
\end{equation} 
Here, $W(\cfg)$ denotes the \emph{classical $\mathcal{W}$-algebra} associated to the Langlands dual algebra $\cfg$. It is defined to be the kernel of certain operators known as \emph{screening operators} and can be considered as the affine analogue of $\Sym(\fh)^W$; see \cite[\S 8]{FrenkelBook} for more details. 

We refer to the algebra homomorphism 
\[
\cF: \fz_c \ra \Sym(\fh^-)
\]
as the \emph{Harish-Chandra homomorphism} for the affine vertex algebra $V(\hfg)$. By the above discussion, it is an injective morphism whose image identifies with the classical $\cW$-algebra associated to $\cfg$. In the language of connections (cf. the next section), the analogue of this homomorphism is also known as the \emph{Miura transformation}.

\subsection{Centre of the completed enveloping algebra} \label{ss:centreAffine}
Next, we turn our attention to the centre $\cZ_c$ of the completed enveloping algebra $\tU_c(\hfg)$ at the critical level. Equivalently, $\cZ_c$ is the (Bernstein) centre of the category of smooth critical level $\hfg_c$-modules.

 \sssec{Segal-Sugawara operators} 
 Let $S_1,\dots, S_n\in \fz_c$ be a complete set of Segal-Sugawara vectors (Definition \ref{d:CM}). The Fourier coefficients $S_{i,[n]}$ are known as \emph{Segal-Sugawara operators}. It follows from the definition that they belong to the centre of $U(V_{c}(\hfg))\subset \tU_c(\hfg)$. Using the Lie algebra homomorphism  \eqref{eq:embedding}, one can show that they are also central in $\tU_c(\hfg)$ \cite{FrenkelBook}. A variant of the above-mentioned theorem of Feigin and Frenkel states that 
 $\cZ_c$ is the completion of a polynomial algebra on infinitely many variables; more precisely, we have 
\begin{equation} \label{eq:FFA2}
\cZ_c \simeq \widetilde{\bC[S_{i,[n]}]_{i=1,\dots, \ell}^{n\in \bZ}}.
\end{equation} 
This completion is defined to be the inverse limit 
\[
\varprojlim_N \left({\bC[S_{i,[n]}]_{i=1,\dots, \ell}^{n\in \bZ}}/I_N\right),
\] where $I_N$ is the ideal generated by $S_{i,[m]}, n\geq mN$. We can think of elements of this completion as infinite linear combinations of $S_{i,[n]}$'s. 
 
  
  \sssec{Action on smooth modules}
We now consider the action of $\cZ_c$ on some smooth modules. 
 Let $V$ be a smooth $\hfg$-module at the critical level. Then, we have a natural homomorphism $\widetilde{\bC[S_{i,[n]}]_{i=1,\dots, \ell}^{n\in \bZ}}\ra \End_{\hfg}(V)$ defined as the composition 
 \[
 \widetilde{\bC[S_{i,[n]}]_{i=1,\dots, \ell}^{n\in \bZ}}\simeq \cZ_c \inj \tU_c(\hfg) \ra \End_{\hfg}(V).
 \]
 Thus, the Segal-Sugawara operators $S_{i,[n]}$ act on smooth critical level modules. 
 Using Proposition \ref{p:estimate}, one can infer some information about this action for some well-known root modules $V$. 

\sssec{Example: Congruence subalgebras} Let $m$ be a positive integer and suppose $r$ is the constant function defined by 
\[
r(\alpha)=m, \quad \quad \forall \, \alpha\in \Phi^*.
\] 
Then 
\[
\fk_r=t^m\fgbb
\]
 is a known as a \emph{congruence subalgebra}. In this case, Proposition \ref{p:estimate} implies that 
 \begin{equation} \label{eq:congruence}
 \textrm{$S_{i,[n]}$ acts trivially on $V_{\fk}$ for all $n\geq d_i.m$.}
 \end{equation} 
  Here $d_1,\dots, d_n$ are the degrees of the fundamental invariants of $\fg$. Statement \eqref{eq:congruence} is a theorem of Beilinson and Drinfeld \cite[\S 3.8]{BD}. 

\sssec{Example: Moy-Prasad subalgebras}
 The above example can be generalised as follows. Let $x$ be a point in the affine building of $G$, and let $r\in \bR_{\geq 0}$. Set 
\[
r(\alpha)=1-\lceil \alpha(x)-r\rceil.
\] 
Then, $\fk_r$ identifies with the Moy-Prasad subalgebra 
\[
\fk_r=\fg_{x,r^+}.
\]
 In this case, Proposition \ref{p:estimate} implies that 
 \begin{equation} 
 \textrm{$S_{i,[n]}$ act trivially on $V_{\fk}$ for all $n\geq (r+1)d_i$.} 
 \end{equation} 
 This result is proved in \cite{CK}. One recovers the previous example by choosing 
 \[
 x=0\quad \textrm{and} \quad r=(m-1).
 \]


\section{Centre at the critical level for $\hgl_n$} \label{s:gln}
 In \S \ref{ss:CM}, we recall Chervov, Molev and Talalaev construction of a complete set of Segal-Sugawara vectors for $\hgl_n$ \cite{CT}, \cite{CM}. In \S \ref{ss:actionCM}, we study the action of the corresponding Segal-Sugawara operators on root modules. In \S \ref{ss:geometric}, we give an alternative definition of these operators using opers. We will then combine these with the vanishing results obtained above to give a proof of the main theorem in \S \ref{ss:proof}.

 \subsection{Segal-Sugawara vectors for $\widehat{\gl_n}$} \label{ss:CM}
 \sssec{Affine Kac-Moody algebra associated to $\hgl_n$}
 In the previous section, we considered the affine Kac-Moody algebra associated to a finite dimensional simple Lie algebra. It is easy to adapt all the definitions and constructions to the case of $\gl_n$. 
 
 Recall that the Killing form for $\gl_n$ is defined by
  \[
 (X,Y)=2n\on{tr}(XY)-2 \on{tr} X \on{tr} Y. 
 \]
 Given a bilinear form $\kappa$ which is a multiple of the Killing form, we can define the affine Kac-Moody algebra $(\hgl_n)_\kappa$  as in \eqref{eq:central}. The critical level will again be when $\kappa$ equals $-\frac{1}{2}$ of the Killing form. The affine vertex algebras $V_\kappa(\hgl_n)$ and the completed universal enveloping algebras $\tU_\kappa(\hgl_n)$ are defined in analogous manner; see \cite{CM} for more details.

  \sssec{Chervov-Molev-Talalaev construction} Following \cite{CM}, we give an explicit construction of a complete set of Segal-Sugawara vectors
   \[
  S_1,S_2,\dots, S_n\in \fz_c(\hgl_n)\subset V_c(\hgl_n).
  \]
   For an arbitrary $n\times n$ matrix $A=[a_{ij}]$ with entries in a ring, its \emph{column determinant} $\cdet$  is defined by 
 \begin{equation}\label{eq:cdet}
 \cdet \, A := \sum_{\sigma} \on{sgn}\, \sigma . a_{\sigma(1)1} \dots a_{\sigma(n)n},
 \end{equation} 
 where the sum is over all permutations $\sigma$ of the set $\{1,2,\dots, n\}$.

 Let $e_{ij}$ be the standard basis elements of $\gl_n$. Let $e_{ij}[r]$ denote the element $e_{ij}\otimes t^r$ of the loop algebra $\gl_n\llp t\rrp$. We will also need the extended affine algebra $\hgl_n \oplus \bC.\tau$ where 
  \begin{equation}\label{eq:tau}
[\tau, e_{ij}[r]]=-re_{ij}[r-1], \quad \quad [\tau, \bone]=0.
\end{equation}
Let $E[-1]$ denote the matrix with $e_{ij}[-1]$ at $ij$ entry. 

\begin{defe} Define $S_1, \dots, S_n\in U(\fg^-)$ by
\begin{equation} \label{eq:CMSS}
\cdet(\tau+E[-1]) = 	\tau^n+\tau^{n-1}S_1 + \cdots + \tau S_{n-1} + S_n.
\end{equation}
\end{defe}

 The main theorem of \cite{CM} is that $S_1,\dots S_n$ is a complete set of Segal-Sugawara vectors for $\hgl_n$ (Definition \ref{d:CM}).  
 
 \sssec{Affine Harish-Chandra homomorphism and Chervov-Molev-Talalaev operators}
 We need to record a result of Chervov and Molev which describes the image of the $S_i$'s under the Harish-Chandra homomorphism 
 \[
 \cF: \fz_c(\hgl_n)\inj \Sym(\fh^-).
 \]
  Define $\omega_i\in \Sym(\fh^-)$ by the equation 
\begin{equation}\label{eq:omega}
 (\tau+E_{11}[-1])\dots (\tau+E_{nn}[-1])= \tau^n +  \tau^{n-1}\omega_1 + \cdots + \omega_n, 
\end{equation} 

Then, according to \cite[\S 6]{CM}, we have 
\begin{equation} \cF(S_i)=\omega_i, \quad \quad i\in \{1,\dots, n\}. \label{p:CM}
\end{equation}

\begin{exam} \label{e:gl2}
Suppose $n=2$. Then we have 
 \begin{multline*}
\cdet (\tau+E[-1]) = (\tau+e_{11}[-1])(\tau+e_{22}[-1]) - e_{12}[-1]e_{21}[-1]=\\
\tau^2+\tau e_{22}[-1] + e_{11}[-1]\tau + e_{11}[-1]e_{22}[-1] - e_{12}[-1]e_{21}[-1]=\\ 
\tau^2 + (E_{22}[-1]+E_{11}[-1]) \tau -e_{11}[-2] + e_{11}[-1]e_{22}[-1] - e_{12}[-1]e_{21}[-1],
 \end{multline*}
 where the last equality is obtained by using the equality $E_{11}[-1]\tau=\tau E_{11}[-1] - e_{11}[-2]$.
 Thus, we have 
\[
 \begin{matrix}
S_1 & =  & E_{11}[-1]+E_{22}[-1]\\
S_2 & =  &  -e_{11}[-2] + e_{11}[-1]e_{22}[-1] - e_{12}[-1]e_{21}[-1].
\end{matrix} 
\]
Furthermore, in this case,
\[
\begin{matrix} 
\omega_1 & = & E_{11}[-1]+E_{22}[-1]\\
\omega_2 & = & -e_{11}[-2] + e_{11}[-1]e_{22}[-1].
\end{matrix} 
 \]
 \end{exam}

\ssec{Action of Chervov-Molev-Talalaev operators on root modules} \label{ss:actionCM}
Our goal is to study the action of the Segal-Sugawara operators defined by Chervov and Molev and Talalaev on certain root modules for $\hgl_n$. We begin by recording a property of Chervov-Molev-Talalaev operators.  

\sssec{A property of Chervov-Molev-Talalaev operators}
Let $\ell\in \{1,\dots, n\}$ and write the vector $S_\ell\in U(\fg_-)$ as a linear combination of elements of the form  
\[
 e_{i_1j_1}[u_1]\dots e_{i_kj_k}[u_k]
 \]
 where $u_i$'s are negative integers whose sum equals $-\ell$.

 \begin{lem} \label{l:cdet}
 At most one element of the set $\{i_1,\dots, i_k\}$ equals $n$. 
 \end{lem} 
 
 \begin{proof} According to \eqref{eq:cdet}, we have 
 \[
 \cdet(\tau+E[-1])=\sum_{\sigma } \on{sgn} \sigma . A_\sigma, \quad \quad A_\sigma=a_{\sigma(1)1} \dots a_{\sigma(n)n}
 \]
  where $a_{ij}=\delta_{i,j}\tau+e_{ii}[-1]$. Obviously, for every $\sigma$, there is only one $i$ with $\sigma(i)=n$. Thus, $A_\sigma$ has exactly one term $a_{ij}$ with $i=n$. According to \eqref{eq:tau}, gathering all the $\tau$'s in $A_\sigma$ to the left preserve elements  of the form $e_{ij}$. The lemma is, therefore, established. \end{proof}

\sssec{Root subalgebra} 
Recall the definition \eqref{eq:pm} of the subalgebra $\fk_m\subset \gl_n(\cO)$ from the introduction. Let $\fk_m^\circ$ denote the root subalgebra of $\hfg$ defined by the function 
\begin{equation}\label{eq:fp0}
r(\alpha_{ij}) = 
\begin{cases} m & i=n \\
0 & j=n, i\neq n\\
1 & \textrm{otherwise} 
\end{cases} 
\end{equation} 

Explicitly, we have 
\begin{equation}\label{eq:km0}
\fk_m^0:=
\begin{bmatrix} 
\cP & \cdots & \cP & \cO  \\
\vdots & \ddots  &  \vdots & \vdots \\
\cP & \cdots & \cP & \cO \\
\cP^m & \cdots &  \cP^m & \cP^m
\end{bmatrix}.
\end{equation}

Note that $\fk_m^0$ is a pro-nilpotent subalgebra of $\fk_m$ and we have $\fk_m/\fk_m^0\simeq \gl_{n-1}(\cO)$; hence, $\fk_m^0$ is the pro-nilpotent radical of $\fk_m$. (We shall not need the last fact).

 \sssec{Action on the corresponding root module} 
 \begin{prop} \label{p:mainEstimate} 
 Let $v$ be a vector in a $\hfg_\kappa$-module annihilated by $\fk_m^\circ$.  Let $S_1,\dots, S_n$ be the  Segal-Sugawara vectors defined above. Then, we have:
 \[
 S_{\ell,[N]}.v=0,\quad \quad  \forall \, \, N\geq m+\ell-1.
 \]
 \end{prop} 

\begin{proof} We know  that $S_{\ell, N}$ is a linear combination of elements of the form 
\begin{equation}\label{eq:comb}
(e_{i_1j_1}[u_1]\dots e_{i_kj_k}[u_k])_{[N]}
\end{equation}
where $k\leq \ell$ and $u_i$'s are negative integers whose sum equals $-\ell$. It is sufficient to show that every such element annihilates $v$.  

In view of definition of $r$ in \eqref{eq:fp0} and Lemma \ref{l:cdet}, we have that 
\[
r(\alpha_{i_1j_1}) + \cdots + r(\alpha_{i_kj_k}) \leq m+k-1
\]
Now if $N\geq m+\ell-1$, then we have 
\[
N\geq m+\ell-1 \geq m+k-1+\ell-k \geq (r(\alpha_{i_1j_1}) + \cdots + r(\alpha_{i_kj_k})) - (u_1+\cdots + u_k) - k.
\]
Thus, according to Proposition \ref{p:estimate}, elements of the form \eqref{eq:comb} annihilated $v$ for all such $N$. 
\end{proof}

\begin{rem}\label{r:false}
Lemma \ref{l:cdet} (and consequently Proposition \ref{p:mainEstimate}) may be false for other sets of Segal-Sugawara vectors. Indeed, if $n\geq 4$, then the above lemma is false for the Segal-Sugawara vectors denoted by $T_{i}$ in \cite{CM}.
\end{rem}

\subsection{Geometric description of the Chervov-Molev-Talalaev operators} \label{ss:geometric}
The isomorphisms \eqref{eq:FFA1} and \eqref{eq:FFA2} are \emph{algebraic} versions of the results of Feigin and Frenkel. We now discuss the geometric version of these results. 

\sssec{The centre in terms of opers}
We defined in the introduction the space of opers $\Op_n(\cDt)$ on the punctured disk. We can easily define a holomorphic analogue of this space; i.e., the space of opers on the disk $\cD$, denoted by $\Op_n(\cD)$. The elements of $\Op_n(\cD)$ are determined by an $n$-tuple 
\[
(a_1,\dots, a_n)\in \cO^n
\]
where $\cO=\bCbb$.

 As mentioned in the introduction, the geometric version of the Feigin-Frenkel Isomorphisms can be formulated as follows: 
 \begin{equation}\label{eq:FFGeometric} 
\fz_c(\hgl_n) \simeq \bC[\Op_n(\cD)],\quad \quad \cZ_c(\hgl_n) \simeq \bC[\Op_n(\cDt)].
\end{equation}

\sssec{Harish-Chandra homomorphism} 
Next, we discuss the geometric reformulation of the Harish-Chandra homomorphism $\cF: \fz_c\ra \Sym(\fh^-)$ . We identify $\fz_c$ with $\bC[\Op_n(\cO)]$ and $\Sym(\fh^-)=\bC[\fhbb]$, using Residue pairing. Thus, we wish to define a map $\cF: \bC[\Op_n(\cD)]\ra \bC[\fhbb]$, or alternatively, a morphism 
\[
\mu: \fhbb \ra \Op_n(\cD) 
\]
whose pullback at the level of algebra of functions equals $\cF$. 

For  
\[
h=(E_{11}(t), \dots, E_{nn}(t))\in \fhbb, \quad \quad E_{ii}(t)\in \bCbb
\]
we define
\begin{equation}
\mu(h):=(a_1,\dots, a_n)
\end{equation} 
where $a_i=a_i(t)\in \bCbb$ is specified by the equation 
\begin{equation} \label{eq:mu}
(\partial_t+E_{11}(t))\dots (\partial_t + E_{nn}(t)) = \partial_t^n + a_{1}(t)\partial_t^{n-1} + \cdots + a_n(t).
\end{equation} 
Here we use the usual commutation relation in the Weyl algebra
\begin{equation}\label{eq:commutation}
[\partial_t, t^n]=-n t^{n-1}
\end{equation} 
to move powers of $\partial_t$ to the right. We now $\mu$ show that the pullback of $\mu$ satisfied the required property. 

\sssec{} Write $a_i(t)=\sum_{k<0} a_{i,k}t^{-k-1}$. Then, for $i=1,\dots, n$ and negative integers $k$, we have functions 
\[
v_{i,k}: \Op_n(\cD)\ra \bC\quad \quad v_{i,k}(a_1,\dots, a_n)=a_{i,k}.
\]

\begin{lem} $\mu^*(v_{i,-1})=\omega_i$ for $i=1,\dots, n$.
\end{lem} 

\begin{proof} By definition, $\mu^*(v_{i,-1})$ is the function on $\fhbb$ defined by  
\[
(E_{11}[t], \dots, E_{nn}(t)) \longmapsto a_{i,-1},
\]
 where  $a_i(t)=\sum_{k<0} a_{i,k}t^{-k-1}$ is defined by \eqref{eq:mu}. (Here, we are using the usual notation $E_{ii}[t]=\sum_{k<0} E_{ii}[-k-1]$.)
Since we are only interested in $a_{i,-1}$, it is clear that we can restrict ourselves to considering 
\[
(\partial+E_{11}[-1])\dots (\partial+E_{nn}[-1]).
\]
Indeed, the commutation relation \eqref{eq:commutation} ensures that the order of poles can only decrease when we convert the product on the LHS of \eqref{eq:mu} to the sum on the RHS. The result now follows from the definition of $\omega_i$'s given in \eqref{eq:omega}. 
\end{proof} 

\begin{exam} Suppose $n=2$. Let $h=(E_{11}(t), E_{22}(t))\in \fhbb$. The morphism $\mu: \fhbb\ra \Op_n(\cO)$ sends $h$ to the oper 
\[
(\partial_t+E_{11}(t))(\partial_t+E_{22}(t)) = \partial_t^2 + (E_{11}(t)+E_{22}(t)-E_{11}(t)')\partial_t + E_{11}(t) E_{22}(t).
\]
Thus, $\mu^*: \bC[\Op_n(\cO)]\ra \bC[\fhbb]$ satisfies 
\[
(\mu^*(v_{1,-1}))(h)= E_{11}[-1]+E_{22}[-1]-E_{11}[-2], \quad \quad (\mu^*(v_{2,-1}))(h) = E_{11}[-1]E_{22}[-1].
\]
In view of Example \ref{e:gl2}, if we identify $\bC[\fhbb]$ with $\Sym(\fh^-)$, we obtain 
\[
\mu^*(v_{1,-1}) = \omega_1, \quad \quad \mu^*(v_{2,-1})=\omega_2. 
\]
\end{exam}

\sssec{Chervov-Molev-Talalaev operators} 
The above discussion allows us to give an alternative description of Chervov and Molev's Segal-Sugawara vectors and the corresponding operators. 

\begin{lem} Under the isomorphism $\fz_c\simeq \bC[\Op_n(\cO)]$, the Chervov-Molev-Talalaev vector $S_i$ maps to $v_{i,-1}$. 
\end{lem} 

\begin{proof} Indeed, the images of both $S_i$ and $v_{i,-1}$ under the map $\cF: \fz_c\simeq  \bC[\Op_n(\cO)] \inj \Sym(\fh^-)$ equals $\omega_i$. 
\end{proof} 

Now let us revisit the isomorphism $\cZ_c \simeq \bC[\Op_n(\cDt)]$. Let $v_{i,k}$ denote elements of $\bC[\Op_n(\cDt)]$ defined as above. Since we are working with \emph{meromorphic} (as oppose to holomorphic) differential equations, $k$ is allowed to be an arbitrary integer. 

\begin{cor}\label{c:opers}
 Under the isomorphism $\cZ_c \simeq \bC[\Op_n(\cDt)]$, $S_{i,[k]}$ maps to a scalar multiple of $v_{i,k}$. 
\end{cor} 

\begin{proof} This follows from the fact that the isomorphism $\cZ_c\simeq \bC[\Op_n(\cDt)]$ intertwines the action of the operator $T$ (the translation operator of the affine vertex algebra $V_c(\hfg)$) with the operator 	$-\partial_t$. We refer the reader to \cite{FrenkelBook}, Proof of Theorem 4.3.2, for the details. 
\end{proof} 

\begin{rem} In \cite[\S 8.2.2]{FrenkelBook}, the Miura transformation is defined to be the map
$\fhbb \ra \Op_n(\cDt)$ specified by the formula 
\begin{equation}\label{eq:Miura2} 
\partial_t +h \mapsto \partial_t+p_{-1}+h,
\end{equation} where $p_{-1} = E_{2,1}+E_{3,2} + \cdots + E_{n,n-1}$. The equivalence of this definition with the one given above is proved in  \cite[\S 3.24]{DS}. See also \cite[\S 4]{FrenkelLecture}.
\end{rem}

\ssec{Proof of the main theorem}\label{ss:proof} As mentioned after the statement of the theorem, the case $m=0$ is well-known. So let us assume that $m>0$.  Suppose $\hfg_c-\modd_\chi$ has a $K_m$-equivariant object. To prove the theorem, it suffices to show that $\chi$ has irregularity less than $m$. 

\sssec{Reduction to a statement in classical representation theory} 
It is more convenient to work with $K_m^0$ since this is a prounipotent group. By assumption $V$  is a $(\hfg, \fk_m^0)$ Harish-Chandra module. As $\fk_m^0$ is pro-nilpotent, this happens if and only if we have a non-trivial homomorphism of $\hfg$-modules 
\[
V_m=V_{\fk_m^0} \ra V.
\]
Hence to prove the theorem, it suffices to show that if an oper $\chi$ acts nontrivially on $V_m$, then $\Irr(\chi)<m$.

\sssec{Action of Chervov-Molev-Talalaev operators} 
Let $S_1,\dots, S_n\in \fz_c(\hgl_n)$ denote the Segal-Sugawara operators defined by Chervov-Molev-Talalaev. Then, we have an isomorphism 
\[
\cZ_c(\hgl_n) \simeq \bC[S_{\ell, [N]}]_{\ell=1,\dots, n}^{N\in \bZ}. 
\]
According to Proposition  \ref{p:mainEstimate}, we have 
\[
S_{\ell,[N]}.v=0,\quad \quad  \forall \, \, N\geq m+\ell-1.
\]

\sssec{Reformulation in terms of opers} 
Let $v_{\ell, N}$ denote the function on $\Op_n(\cDt)$ corresponding to $S_{\ell,[N]}$ under the Feigin-Frenkel Isomorphism. Then, by Corollary \ref{c:opers}, we have  
\[
v_{\ell, [N]}.v=0,\quad \quad N\geq m+\ell-1
\]

We can translate this result as follows: let $\chi=(a_1,\dots, a_n)$ denote an oper acting nontrivially on $V_m$. Then, we must have  
\[
-\nu(a_\ell)\leq  m+\ell-1
\]
where  $-\nu(a_\ell)$ denotes the order of the pole of $a_\ell \in \cK=\bCpp$. 

\sssec{Using Komatsu-Malgrange formula to estimate the irregularity} 
The previous inequality implies that
\[
\ell-\nu(a_{n-\ell}) \leq \ell + (m + n-\ell -1)\leq m+n-1.
\]
According to the Komatsu-Malgrange formula \eqref{eq:Irr}, this implies that $\Irr(\chi)\leq m-1$. The theorem is, therefore, established. 
\qed

  \end{document}